\documentclass[sn-chicago,Numbered]{sn-jnl}

\usepackage{graphicx}%
\usepackage{multirow}%
\usepackage{amsmath,amssymb,amsfonts}%
\usepackage{amsthm}%
\usepackage{mathrsfs}%
\usepackage[title]{appendix}%
\usepackage{xcolor}%
\usepackage{textcomp}%
\usepackage{manyfoot}%
\usepackage{booktabs}%
\usepackage{algorithm}%
\usepackage{algorithmicx}%
\usepackage{algpseudocode}%
\usepackage{listings}%
\usepackage{graphicx}
\usepackage{a4wide}
\usepackage{orcidlink}
\usepackage{amscd}
\usepackage{verbatim}
\usepackage[mathscr]{eucal}
\usepackage[all]{xy}
\usepackage{fdsymbol}
\usepackage{mathtools}
\usepackage{stmaryrd}
\usepackage{fancyvrb}
\usepackage{tikz-cd}
\usepackage{enumitem}
\usepackage{url}

\newcommand{\calA}{\mathcal{A}}

\newcommand{\calF}{\mathcal{F}}

\newcommand{\frakg}{\mathfrak g}

\newcommand{\be}{\begin{equation}}
\newcommand{\ee}{\end{equation}}
\newcommand{\lamrodd}{\Lambda_{\ell,\text{reg}}^\text{odd}}
\newcommand{\lamrev}{\Lambda_{\ell,\text{reg}}^\text{even}}

\newcommand{\Tr}{\text{Tr}}

\newcommand{\Z}{\mathbb{Z}}
\newcommand{\Q}{\mathbb{Q}}

\newcommand{\ds}{\displaystyle}
\allowdisplaybreaks

\theoremstyle{plain}
\newtheorem{theorem}{Theorem}[section]

\newtheorem{lemma}[theorem]{Lemma}
\newtheorem{proposition}[theorem]{Proposition}

\theoremstyle{definition}

\newtheorem{example}[theorem]{Example}
\newtheorem{remark}[theorem]{Remark}

\theoremstyle{remark}

\begin{document}


\title[Article Title]{$\ell$-adic properties and congruences of $\ell$-regular partition functions}

\author{\fnm{Ahmad} \sur{El-Guindy}$^{1*}$\orcidlink{0000-0002-8581-5465}}\email{ aelguindy@sci.cu.edu.eg}

\author{\fnm{Mostafa} \sur{M. Ghazy}$^{2}$\orcidlink{0009-0003-0421-6079}}\email{ mostafa\_ghazy@science.tanta.edu.eg}



\affil{$^{1*}$\orgdiv{Department of Mathematics}, \orgname{Faculty of Science, Cairo University}, \orgaddress{ \city{Giza}, \postcode{12613}, \state{Cairo}, \country{Egypt}}}

\affil{$^{2}$\orgdiv{Department of Mathematics}, \orgname{Faculty of Science, Tanta University}, \orgaddress{\city{Tanta}, \postcode{31527}, \state{Tanta}, \country{Egypt}}}


\abstract{We study $\ell$-regular partitions by defining a sequence of modular forms of level $\ell$ and quadratic character which encode their $\ell$-adic behavior. We show that this sequence is congruent modulo increasing powers of $\ell$ to level $1$ modular forms of increasing weights. We then prove that certain $\Z/\ell^m\Z$-modules generated by our sequence are isomorphic to certain subspaces of level $1$ cusp forms of weight independent of the power of $\ell$, leading to a uniform bound on the ranks of those modules and consequently to $\ell$-adic relations between $\ell$-regular partition values.}



\maketitle


\section{Introduction}

The study of the partition function and its variants has been a source of inspiration for many number theorists. A partition of a nonnegative integer $n$ is a non-ascending sequence of positive integers that add up to $n$. Letting $p(n)$ denote the number of partitions of $n$, we have the generating function 

\begin{equation}
\sum_{n=0}^{\infty}p(n)q^n=\prod_{n=1}^{\infty}\frac{1}{1-q^n}.
\end{equation}

The Ramanujan congruences and their prime power extensions, which were established by Atkin\cite{ATK}, Ramanujan \cite{Ramanujan1921} and Watson \cite{Wat}, are some of the most fundamental and elegant arithmetic properties of $p(n)$. Suppose $\ell\geq 5$ is  prime, and for $m\geq 1$ define $\delta_\ell(m)$ by $24\delta_\ell(m)\equiv 1\pmod{\ell^m}$ and $1\leq \delta_\ell(m)\leq \ell^m-1$. Then, for all $n\geq 0$ we have 

\begin{align}\label{ram}
\begin{split}  
&p(5^mn+\delta_5(m))\equiv 0\pmod{5^m},\\
&p(7^mn+\delta_7(m))\equiv 0\pmod{7^{\lfloor m/2\rfloor+1}},\\
&p(11^mn+\delta_{11}(m))\equiv 0\pmod{11^m}.
 \end{split}
\end{align} 

It is natural to seek extensions of those remarkable congruences, as well as attempt to place them in a wider context. Fundamental work by Ahlgren and Ono \cite{AhlDis,AhlOnoCon,OnoDis} proved the existence of infinitely many non-nested arithmetic progressions $An+B$ where $p(An+B)\equiv 0\pmod{M}$ for fixed $M$ coprime to 6. More recently, Folsom, Kent, and Ono incorporated the theory of $\ell-$adic modular forms in the sense of Serre as in \cite{SJP} to establish a general framework for partition congruences \cite{FKO}, unveiling additional congruences such as 

\be\label{p13}
p(13^4n+27,371)\equiv 25p(13^2n+162)\pmod{13^2}.
\ee

To be more specific, they defined a sequence of functions $L_\ell(b,z)$ encoding $\ell$-adic properties of $p(n)$, and they relate them to certain finite dimensional spaces of modular forms of level 1. The cases where those spaces are of dimension $0$ correspond to \eqref{ram}, while other cases lead to  extended and systematic families of congruences such as \eqref{p13}. The work of \cite{FKO} was elaborated on by Boylan and Webb in \cite{BMJ}, and it was also extended to the study of $r$-colored partitions and Andrews' spt function in \cite{BelmontLee}. 

In the present work, we wish to extend those results and techniques to study $\ell$-regular partition functions, which we shall recall their definition below. This also requires the extension of results from \cite{FKO,BMJ} to the setting of modular forms with quadratic Nebentypus.

\section*{Acknowledgements}
 The authors would like to sincerely thank the anonymous referee for a thorough reading of the initial version of this work and for their helpful and thoughtful suggestions which significantly improved the presentation and content of the paper.  We are also truly indebted to Ken Ono for several useful discussions. In addition, A. El-Guindy is grateful to the Abdus Salam International Centre for Theoretical Physics for its warm hospitality through the Associates Programme 2019-2024 which was instrumental for finalizing this manuscript.  

\subsection{$\ell-$Regular partition functions}

A partition of an integer $n$ is called $k$-regular if none of its parts is divisible by $k$. Let $b_k(n)$ denote the number of $k$-regular partitions of $n$ (setting $b_k(0)=1$ by convention), then we have the generating function

\[
\ds \sum_{n=0}^\infty b_k(n)q^n=\prod_{m=1}^\infty\frac{ (1-q^{km})}{ (1-q^m)}.
\]

For prime $\ell$, $b_\ell(n)$ is the number of irreducible $\ell$-modular representations of the symmetric group $S_n$. There has been numerous studies on the arithmetic of $b_k(n)$ modulo $m$  for various values of $k$ and $m$, see \cite{XEX,AB,Ray21,7-regular,CAlDraPen,DanPen,FurPen,GorOno,Penn,LovPen,Webb} to name just a few. Broadly speaking, some of those works address cases where $k$ and $m$ are coprime, while others address the case where $k$ and $m$ are not coprime, and in particular the case where $k=m=\ell$ is itself prime. As examples of the latter, Dandurand  and Penniston \cite{DanPen}, using the theory of complex multiplication, determined exact criteria for the $\ell$-divisibility of $b_\ell(n)$ for $\ell \in \{5,7,11\}$, whereas Xia \cite{XEX}, using theta function identities of Ramanujan, obtained congruences of the form

\[
b_\ell(A(k)n+B(k))\equiv C(k)b_\ell(n) \pmod \ell,
\]
for $\ell\in \{13, 17,19\}$ and certain functions $A(k), B(k), C(k)$ depending on $\ell$ and $k$. In earlier work, Gordon and Ono \cite{GorOno} showed that, in an asymptotic statistical sense, $b_k(n)$ is expected to be a multiple of any power of $p$ for every prime $p\mid k$.

For the remainder of the paper, let $\ell\geq 5$ be prime. The main aim of the present work is to provide systematic congruences for all $b_\ell$ modulo $\ell$ for certain arithmetic progression in a manner similar to the work of Folsom, Kent, and Ono in \cite{FKO}, Boylan and Webb \cite{BMJ} and Belmont et al. in \cite{BelmontLee}. The general structure of the arguments is rather similar, yet there are a few important differences that we shall highlight as they arise. In an informal sense,  some of our arguments amount to ``taking a square root" of certain results and arguments in those works. One instance of this phenomenon is that weights of modular forms we study are half  their counterparts in \cite{FKO, BMJ, BelmontLee}, and generally speaking they live in spaces with quadratic Nebentypus.

Our work depends on studying the sequence of generating functions, defined for $b\geq 0$ by

\be\label{bdef}
B_\ell(b;z):=\sum_{n=0}^\infty b_\ell\left(\frac{\ell^b n -\ell+1}{24}\right)q^{\frac{n}{24}},
\ee
where by convention  $b_\ell(0)=1$ and $b_\ell(\alpha)=0$ whenever $\alpha$ is anything but a nonnegative integer. Note that since $\ell^{2m+1}\equiv \ell \pmod{24}$ and $\ell^{2m}\equiv 1 \pmod{24}$, we could identify the equivalence classes of $n$ for which  $\frac{\ell^b n-\ell+1}{24}$ is an integer and hence  $b_\ell\left(\frac{\ell^a n-\ell+1}{24}\right)$ doesn't vanish. For notational convenience,   we attach an integer $c$ to primes $\ell\geq 5$ as follows

\[
c=c_\ell:=24\left\lceil\frac{\ell-1}{24}\right\rceil-(\ell-1)
\]
(so that $c+\ell-1\equiv 0 \pmod{24}$ and $0\leq c<24$). It follows that

\[
B_\ell(b;z)=\begin{cases}
\sum_{n=0}^\infty b_\ell\left(\ell^a n+\frac{\ell^a c-\ell+1}{24}\right)q^{n+\frac{c}{24}} \textrm{ if $b$ is odd}, \\
\sum_{n=1}^\infty b_\ell\left(\ell^a n-\frac{\ell^a c+\ell-1}{24}\right)q^{n-\frac{c}{24}}\textrm{ if $b$ is even}.  
\end{cases}
\]

Let $\eta(z)=q^{\frac{1}{24}}\prod_{n=1}^\infty (1-q^n)$ be the standard Dedekind eta function and set

\[
\Phi_\ell(z)=\frac{\eta(\ell^2 z)}{\eta(z)}.
\]
Recall that the Atkin $U(\ell)$-operator acts on $q$-series by 
 
\begin{equation}\label{udef}
\ds \left(\sum a(n)q^n\right)\mid U(\ell):=\sum a(n\ell)q^n.
\end{equation}
We also set $R_\ell(0;z)=\eta(\ell z)\eta(z)^{c-1}$ and define a sequence of functions inductively for $b\geq 1$ by 

\[
R_\ell(b;z)=\begin{cases}
R_\ell(b-1;z) \Phi_\ell^c(z)\mid U(\ell) \textrm{ if $b$ is odd}, \\
R_\ell(b-1;z)|U(\ell) \textrm{ if $b$ is even}.  
\end{cases}
\]
A straightforward computation shows that the sequence $R_\ell$ is closely related to $B_\ell$ as follows.

\begin{equation}
\label{Bell}
  R_\ell(b;z)=\begin{cases}
B_\ell(b;z) \eta(\ell z)^c \textrm{ if $b$ is odd}, \\
B_\ell(b;z)\eta(z)^c \textrm{ if $b$ is even}.  
\end{cases}  
\end{equation}

Consider the following two infinite families of descending $\Z/\ell^m\Z$-modules, where $m$ is a positive integer:

\[
\lamrodd(2b+1,m):=\text{Span}_{\Z/\ell^m\Z}\{R_\ell(2\beta+1;z) \pmod {\ell^m}: \beta\geq b\},
\]

\[\lamrev(2b,m):=\text{Span}_{\Z/\ell^m\Z}\{R_\ell(2\beta;z) \pmod {\ell^m}: \beta\geq b\}.
\]
The following result shows that, despite being initially defined by infinitely many generators, these modules stabilize into (isomorphic) finitely generated $\Z/\ell^m\Z$-modules.
\begin{theorem}\label{thm1}
Let $\ell\geq 5$ be prime. For every positive integer $m$ there exists an integer $\mathfrak{b}_\ell(m)$ satisfying the following.
\begin{enumerate}
 \item The nested sequence of $\Z/\ell^m \Z$-modules
 
\begin{equation}\label{nestodd}    
\lamrodd(1,m)\supseteq \lamrodd(3,m)\supseteq \cdots \supseteq \lamrodd(2b+1,m)\supseteq \cdots
\end{equation}
stabilizes for all $b$ with $2b+1\geq \mathfrak{b}_\ell(m)$. Moreover, if we denote the stabilized $\Z/\ell^m \Z$-module by $\Omega_{\ell,\text{reg}}^\text{odd}(m)$ then its rank is bounded above by 

\begin{equation}\label{rankbound}
    \left\lfloor\frac{\ell-1}{12}\right\rfloor-\left\lfloor\frac{\ell-1}{24}\right\rfloor
\end{equation}
\item The nested sequence of $\Z/\ell^m \Z$-modules

\[
\lamrev(0,m)\supseteq \lamrev(2,m)\supseteq \cdots \supseteq \lamrev(2b,m)\supseteq \cdots
\]
stabilizes for all $b$ with $2b\geq \mathfrak{b}_\ell(m)$. Moreover, if we denote the stabilized $\Z/\ell^m \Z$-module by $\Omega_{\ell,\text{reg}}^\text{even}(m)$ then  $\Omega_{\ell,\text{reg}}^\text{even}(m)\cong \Omega_{\ell,\text{reg}}^\text{odd}(m)$.
\end{enumerate}

\end{theorem}

\begin{remark}
A natural question, following the establishment of the fact that the nested sequences of modules in Theorem \ref{thm1} stabilize, is to determine the rate of such stabilization. This seems to be a subtle question which isn't fully settled, even in the case of $p(n)$ as addressed in \cite{FKO} and in more details in \cite{BMJ}. Indeed, utilizing similar methods to those papers, we state in section 5 (after having introduced some more notation and tools) bounds for $\mathfrak{b}_\ell(m)$ in terms of dimensions of certain spaces of modular forms (cf. Remark \ref{boundrmrk}).
\end{remark}


Our next result states that for certain small primes, the ranks of the stabilized modules in Theorem \ref{thm1} are $\leq 1$, leading immediately to interesting congruences modulo powers of $\ell$. 
\begin{theorem}
\label{mainth}
 Let $5\leq \ell \leq 31$ be prime, and let $m\geq 1$. There exists positive integers $\mathfrak{b}_\ell(m)$ such that if $b_1\equiv b_2 \pmod{2}$ are integers for which $b_2> b_1 \geq \mathfrak{b}_\ell(m)$, then there is an integer $\mathfrak{B}_\ell(b_1,b_2;m)$ such that for every nonnegative integer $n$ we have 
 
\begin{align*}
 b_\ell\left(\ell^{b_1} n+\frac{\ell^{b_1} c-\ell+1}{24}\right)\equiv \mathfrak{B}_\ell(b_1,b_2;m)b_\ell\left(\ell^{b_2} n+\frac{\ell^{b_2} c-\ell+1}{24}\right)\pmod{\ell^m}  \text{ if $b_1$ is odd},\\
 b_\ell\left(\ell^{b_1} n-\frac{\ell^{b_1} c+\ell-1}{24}\right) \equiv \mathfrak{B}_\ell(b_1,b_2;m) b_\ell\left(\ell^{b_2} n-\frac{\ell^{b_2} c+\ell-1}{24}\right)\pmod{\ell^m}\text{ if $b_1$ is even}.
\end{align*}
Furthermore, if $\ell \in \{5,7,11\}$, then $\mathfrak{B}_\ell(b_1,b_2;m)=0$.
\end{theorem}
\begin{proof}
 For $\ell \in \{5,7,11\}$ the bound \eqref{rankbound} is $0$, whereas for $13\leq \ell \leq 31$ that bound is $1$.   
\end{proof}

\begin{example}
We illustrate Theorem \ref{mainth} with $\ell= 17$. For $m=1$, Theorem \ref{mainth} applies for every pair of positive integers $b_1<b_2$ with the same parity. We let $b_1:=1$ and $b_2:=3$. It turns out that $\mathfrak{B}(1,3;1)=14$, and so for all $n\geq 0$ we have 
\begin{equation}\label{ex17}
b_{17}(17n+5)\equiv 14~b_{17}(17^3 n+1637)\pmod{17}.
\end{equation}
Indeed the first few cases of \eqref{ex17} can be verified from $(b_{17}(5),b_{17}(22), b_{17}(39))=(7,995,30176)$, and 
\begin{tiny}
    \begin{align*}
    b_{17}(1637)&=16073386675530933163774672019535658494433,\\
    b_{17}(6550)&=309851360589797865267403451823810081019783734795785160796304968259966951502243927905,\\
    b_{17}(11463)&= 3560475442109243873958783824257503520544859191439134627497615645463903848278620702913653368340826627325667207616.
\end{align*}
\end{tiny}
\end{example}


\section{\bf Modular properties of $R_\ell(b;z)$ and related functions}

\subsection{Basic definitions}

For a positive integer $N$, a function of the form $\ds f(z)=\prod_{\delta \mid N}\eta(\delta z)^{r_\delta}$, where $r_\delta \in \Z$, is called an \emph{eta-quotient}. Such eta-quotients will figure repeatedly in our work, and the following well-known result (see for example Theorems 1.64 and 1.65 of \cite{OWEB}) describes their modularity properties.

\begin{lemma}\label{etaquo}
Let $N$ be a positive integer and $r_\delta \in\Z$ be such that

\[
\sum_{\delta\mid N}\delta r_\delta \equiv \sum_{\delta\mid N}\frac{N}{\delta} r_\delta\equiv 0 \pmod{24}.
\]
If  $f(z)=\ds \prod_{\delta \mid N}\eta(\delta z)^{r_\delta}$,
then for every $\left(\begin{smallmatrix}
    a&b\\
    c&d
\end{smallmatrix}\right) \in \Gamma_0(N),$ we have $f\left(\frac{az+b}{cz+d}\right)=\chi(d)(cz+d)^kf(z)$, where $k=\frac{1}{2}\sum_{\delta|N}r_\delta$ and 

\[
\chi(d)=\left(\frac{(-1)^k\prod_{\delta\mid N} \delta^{r_\delta}}{d}\right).
\]

Furthermore, If $c,d$ are coprime positive integers with $d\mid N$ then the order of vanishing of $f$ at the cusp $\frac{c}{d}$ is given by 

\begin{equation}\label{cuspord}  
\textrm{ord}_{\frac{c}{d}}(f)=\frac{N}{24}\sum_{\delta\mid N}\frac{\text{gcd}(d,\delta)^2r_\delta}{gcd(d,\frac{N}{d})d\delta}.
\end{equation}

\end{lemma}

It follows from Lemma \ref{etaquo} that the character corresponding to $R_\ell(0;z)$ is   

\[
\chi_\ell(d)=\left(\frac{(-1)^{\frac{c}{2}}\ell}{d}\right).
\]
By quadratic reciprocity, we see that $\chi_\ell$ in fact coincides with the Legendre symbol $\left(\frac{\bullet}{\ell}\right)$. When the context is clear, we shall omit the subscript on $\chi_\ell$.

Modifying the construction in  \cite{FKO} to fit the present study, we define the operator $D_c(\ell)$ by

\begin{equation}
f\mid D_c(\ell):=(f\cdot \Phi_\ell^c(z))\mid U(\ell).
\end{equation}
To study congruence properties of $\ell$-regular partitions we will consider the alternating action by the operators $U(\ell), D_c(\ell)$. It will sometimes be useful  to think of this sequence as simply repeated use of the operators $U(\ell)\circ D_c(\ell)$ or $D_c(\ell)\circ U(\ell)$. To this end, we define

\begin{equation}
f\mid X_c(\ell):=(f\mid U(\ell))\mid D_c(\ell)~, \text{and}
\end{equation}

\begin{equation}
    f\mid Y_c(\ell):=(f\mid D_c(\ell))\mid U(\ell). \
\end{equation}

We now prove a result about the modularity of the function $R_\ell(b;z)$. Using standard notation, we denote by $M_k(\Gamma_0(N),\chi)$ (resp. $S_k(\Gamma_0(N),\chi)$) the space of holomorphic modular forms (resp. cusp forms) of weight $k$ on $\Gamma_0(N)$ and character $\chi$, and we denote by $M_k^!(\Gamma_0(N),\chi)$ the space of weakly holomorphic modular forms of weight $k$ on $\Gamma_0(N)$ and character $\chi$, i.e. those forms are whose poles (if any) are supported at the cusps of $\Gamma_0(N)$.

\begin{lemma}
    For all $b\geq 0$, $R_\ell(b;z)\in M_{\frac{c}{2}}^!(\Gamma_0(\ell),\chi_\ell)\cap \mathbb{Z}\llbracket q \rrbracket$.
\end{lemma}

\begin{proof}
    For $b=0$, we have that $R_\ell(0;z)=\eta(\ell z)\eta(z)^{c-1}\in M_{\frac{c}{2}}^!(\Gamma_0(\ell),\chi_\ell)\cap \mathbb{Z}\llbracket q \rrbracket$. It also follows that $\Phi_\ell(z)\in M_0^!(\Gamma_0(\ell^2))\cap \Z\llbracket q \rrbracket$. From its definition we easily see that $U(\ell)$ preserves $\Z\llbracket q \rrbracket$, and it is also well-known that 
    \begin{equation}\label{ulsq}
        U(\ell):M_k^!(\Gamma_0(\ell^2),\chi)\to M_k^!(\Gamma_0(\ell),\chi),
    \end{equation}
and 
\begin{equation}
U(\ell):M_k^!(\Gamma_0(\ell),\chi)\to M_k^!(\Gamma_0(\ell),\chi).    
\end{equation}    
    Combining the facts above for $k=\frac{c}{2}$, an inductive argument show that $R_\ell(b;z)$ is in $M_k^!(\Gamma_0(\ell),\chi_\ell)\cap \mathbb{Z}\llbracket q \rrbracket$ for all $b\geq 1$.
\end{proof}  
Recall that for  $f(z)\in M_k^!(\Gamma_0(N))$ and $\ell\nmid N$ the Hecke operator $T(\ell,k)$ is defined by  
    \begin{equation}
        f(z)\mid T(\ell,k)=f(z)\mid U(\ell)+\ell^{k-1}f(z)\mid V(\ell).
    \end{equation}
It is well-known that $T(\ell,k)$ preserves $M_k^!(\Gamma_0(N),\chi)$ (as well as $M_k(\Gamma_0(N),\chi)$ and $S_k(\Gamma_0(N),\chi)$).  Hence we obtain the following stability properties of $U(\ell)$ modulo $\ell^{k-1}$.
\begin{proposition}
    \label{Stab}
    Let $\ell\geq5$ be prime, and let $f(z)\in M_k^!\cap\mathbb{Z}_{(\ell)}(( q ))$.
    \begin{enumerate}
        \item We have $f(z)\mid T(\ell,k)\equiv f(z)\mid U(\ell)\pmod{\ell^{k-1}}$.
        \item The operator $U(\ell)$ stabilizes $M_k^!\cap\mathbb{Z}_{(\ell)}((q ))$ modulo $\ell^{k-1}$.
    \end{enumerate}
    
\end{proposition}

\subsection{Filtration}
The theory of filtration provides many remarkable and deep results in the study of modular forms. It enables a systematic study of the phenomena of a modular form with $\ell$-integral coefficients, usually of higher level and possibly with Nebentypus, being congruent $\pmod \ell$ to modular forms of level $1$. As we shall see, this applies to our sequence of interest $R_\ell(b;z)$.  In order to proceed, we review some of the most pertinent facts, following \cite{SJP} and \cite{SWH}.
For $0\neq f\in \mathbb{Z}_{(\ell)}\llbracket q \rrbracket$, define the filtration of $f$ modulo $\ell$ by 

\begin{equation}
w_\ell(f):=~{\text {inf}}_{k\geq 0}\{k:f\equiv g\pmod\ell ~{\text {for some}}~ g\in M_k\cap \mathbb{Z}_{(\ell)}\llbracket q \rrbracket\}.
\end{equation}
If $f\equiv 0\pmod{\ell}$ then define $w_\ell(f)=-\infty$. Note that if $f\equiv g \pmod{\ell}$ and $g\in M_k
$, then we have $w_\ell(f)\equiv k \pmod{\ell-1}$. 

The following is part of Lemme 2 on p. 213 of \cite{SJP}; it gives a bound of the filtration of the action of  $U(\ell)$ on $q$-series.

\begin{lemma}
    \label{Uoperator}
     If $f\in \mathbb{Z_{(\ell)}}\llbracket q \rrbracket$ is congruent modulo $\ell$ to a nonzero modular form in some $M_k$ of level $1$ (i.e. $w_{\ell}(f)\geq 0$), then  we have  

 \[
w_\ell(f\mid U(\ell))\leq \ell+\frac{(w_\ell(f)-1)}{\ell}.
\]
\end{lemma}

In the case of $R_\ell(0;z)$, we shall see that  applying $U(\ell)$ and $D_c(\ell)$ decreases the filtration. To describe this precisely we prove the following lemma.


\begin{lemma}
Let $f\in \mathbb{Z}_{(\ell)}\llbracket q \rrbracket$ be congruent to a nonzero modular form modulo  $\ell\geq 5$. Set $k_\ell=\frac{1}{2}(c\ell +\ell-1)$, and suppose that

\label{baselem}
    
    \be \label{cong1} w_\ell(f)\equiv 12~\left\lceil\frac{\ell-1}{24}\right\rceil \pmod{\ell-1}.\ee
    \begin{enumerate}[label=(\alph*)]
        \item If $w_\ell(f)\leq k_\ell$, then we have $w_\ell(f\mid U(\ell))\leq\frac{1}{2} (\ell -1)+\frac{c}{2}$ and $w_\ell(f\mid X_c(\ell))\leq k_\ell$.
        \item If $w_\ell(f)= \frac{1}{2}(\ell -1)+\frac{c}{2}$, then we have $w_\ell(f\mid D_c(\ell))\leq k_\ell $.
        \item If $w_\ell(f)> k_\ell$, then we have $w_\ell(f\mid X_c(\ell))< w_\ell(f)$.
    \end{enumerate}

\end{lemma}

\begin{proof} 
To prove the first part of (a), we utilize Lemma \ref{Uoperator} and the fact that $w_\ell$ is always either $-\infty$ or an integer (actually an even nonnegative integer, but we won't need that detail)  to get

\begin{align*}
   w_\ell\left(f\mid U(\ell)\right)&\leq \left\lfloor \ell+\frac{\frac{1}{2}(c\ell +\ell-1)-1}{\ell}\right\rfloor \\
   &=\left\lfloor \ell+\frac{c}{2}+\frac{\ell-3}{2\ell}\right\rfloor =\ell+\frac{c}{2}.
\end{align*}
Note that we also have $w_\ell(f|U(\ell)) \equiv w_\ell(f) \equiv 12~\lceil\frac{\ell-1}{24}\rceil \pmod{\ell-1}$. The maximal positive integer $\leq \ell +\frac{c}{2}$ satisfying that congruence is easily seen to be $\frac{1}{2} (\ell -1)+\frac{c}{2}$, and the result follows. The second inequality of (a) follows from the first one and from part (b), which we now turn our attention to. Utilizing the fact that  $\Phi^c_\ell(z)\equiv \Delta^\frac{c(\ell+c-1)}{24} \pmod \ell$, and the latter is in $S_{c\left(\frac{\ell^2-1}{2}\right)}$, by a similar reasoning to the above we get 
 
\begin{align*}
   w_\ell\left(f\mid X_c(\ell)\right)&\leq \left\lfloor \ell+\frac{\frac{1}{2}(\ell-1+c+c(\ell^2-1))-1}{\ell}\right\rfloor \\
   &=\left\lfloor \ell+\frac{c}{2}\ell+\frac{\ell-3}{2\ell}\right\rfloor =\ell\left(1+\frac{c}{2}\right).
\end{align*}
Since $w_\ell(\Phi_\ell^c)=\frac{c}{2}(\ell^2-1)\equiv 0 \pmod{\ell-1}$, we see that $w_\ell(f|X_c(\ell)) \equiv w_\ell(f) \pmod{\ell-1}$. The result follows since the maximal positive integer $\leq \ell\left(1+\frac{c}{2}\right)$ satisfying the congruence \eqref{cong1} is indeed $k_\ell$. Part (c) follows in a similar manner, and we omit the details for brevity. 


\end{proof}

We conclude this section with a simple lemma that enables the lifting of congruences modulo $\ell$ to congruences modulo $\ell^j$ by taking appropriate powers.

\begin{lemma}\label{lemmaCon1}
    Suppose that $f(z)\in \mathbb{Z}_{(\ell)}\llbracket q \rrbracket$ has $f(z)\equiv 1\pmod{\ell}$. Then for all $j\geq 1$, we have $f(z)^{\ell^{j-1}}\equiv1\pmod{\ell^j}$.
\end{lemma}


\section{\bf Eisenstein series of level $\ell$ and Legendre character}

In this section, we recall a pair of Eisenstein series that we shall denote $E_{k,\chi}$ and $F_{k,\chi}$ of level $\ell$ and quadratic character. Boylan studied the former in \cite{BoySD} (where it was denoted by $E_{\ell, k, \chi}$) using methods from \cite{Scho} with ideas going back to Hecke \cite{Hecke27}. Further details about those Eisenstein series and the methods used to study them can be found in \cite{fcmf}. Whenever $k\geq 2$ has the same parity as $\chi$, the Eisenstein subspace in $M_k(\Gamma_0(\ell), \chi)$ has dimension $2$ (the number of cusps of $\Gamma_0(\ell)$), and $E_{k,\chi}, F_{k,\chi}$ form a basis for it. Let  the generalized Bernoulli numbers $B_{k,\chi}$ be defined by the generating series

\[
\sum_{a=1}^\ell \chi(a) \frac{te^{at}}{e^{\ell t}-1}=\sum_{k=0}^\infty B_{k,\chi}\frac{t^k}{k!}.
\]
Then we have

\begin{align}
E_{k,\chi}(z)&:=1-\frac{2k}{B_{k,\chi}}\sum_{n=1}^\infty\left(\sum_{d\mid n, d>0}\chi(d)d^{k-1}\right)q^n,\\
    F_{k,\chi}(z)&= \sum_{n=1}^\infty n^{k-1}\left(\sum_{d\mid n, d>0}\chi(d)d^{1-k}\right)q^n \in M_k(\Gamma_0(\ell),\chi_\ell).
\end{align}

For the quadratic  character $\chi=\chi_\ell$ we also recall the following evaluation connecting the value of the Dirichlet $L$-function at a posititve integer $k$ to the generalized Bernoulli numbers $B_{k,\chi}$  

\[
L(k,\chi)=\frac{-\frakg_\chi (-2\pi i)^k B_{k,\chi}}{2\ell^k k!},
\]

where $\frakg_\chi$ denotes the Gauss sum sum corresponding to $\chi=\chi_\ell$, which is defined by 

\[
\mathfrak{g}_\chi=\sum_{n=0}^{\ell-1}\chi(n)e^\frac{2\pi i n}{\ell}.
\] 

It is well-known (see Chapter 6 of \cite{IrRo} for instance) that

\[
\frakg_\chi=\begin{cases}
	\sqrt{\ell} \ \ \ \  \textrm{ if $\ell\equiv 1 \pmod 4$}, \\
	i\sqrt{\ell} \ \ \ \  \textrm{ if $\ell\equiv 3 \pmod 4$}.  
\end{cases}
\]



\begin{example}
    When $\ell=19$, $k=9$ we have
    \[
    E_{9,\chi_{19}}=\frac{19}{3708443635}\left(\frac{3708443635}{19}+q-255q^2-6560q^3+65281q^4+\dots\right)\equiv 1 \pmod{19},
    \]
    and
    \[
    F_{9,\chi_{19}}=q+255q^2+6560q^3+65281q^4+390626q^5+\dots.
    \]
\end{example}

We now consider the action of the Atkin-Lehner involution on $E_{k,\chi}, F_{k,\chi}$. Recall that if $f$ is any function on the upper half-plane and $\gamma=\left(\begin{smallmatrix}
    a&b\\
    c&d
\end{smallmatrix}\right)$ is a real matrix with positive determinant, then the weight $k$ slash action of $\gamma$ on $f$ is defined by

\be
(f|_k\gamma)(z):=(ad-bc)^{\frac{k}{2}}(cz+d)^{-k}f\left(\frac{az+b}{cz+d}\right).
\ee
The Atkin-Lehner operator corresponds to the action of $W_\ell:=\left(\begin{smallmatrix}
    0&-1\\
    \ell&0
\end{smallmatrix}\right).
$
Note that for $\left(\begin{smallmatrix}
    a&b\\
    \ell c&d
\end{smallmatrix}\right) \in \Gamma_0(\ell)$, we have 

\[
W_\ell \left(\begin{smallmatrix}
    a&b\\
    \ell c&d
\end{smallmatrix}\right) =\left(\begin{smallmatrix}
    d&-c\\
    -\ell b &a
\end{smallmatrix}\right)W_\ell.
\]
It follows that if $\psi$ is any character on $(\Z/\ell)^*$ and $f\in M_k(\Gamma_0(\ell), \psi)$, then $f|_kW_\ell \in M_k(\Gamma_0(\ell),\overline{\psi})$. In particular, when $\psi$ is either the trivial character or the quadratic character $\chi_\ell$, the space $M_k(\Gamma_0(\ell), \psi)$ is preserved by $W_\ell$.

Let $\gamma_1,\dots, \gamma_{\ell+1}$ be distinct coset representatives of SL$_2(\Z)$ over $\Gamma_0(\ell)$. Then the trace operator $\Tr:M_k^!(\Gamma_0(\ell))\rightarrow M_k^!$ is defined by

\be\label{trdef1}
\Tr(f):=\sum_{j=1}^{\ell+1}f|_k\gamma_j.
\ee
Using explicit coset representatives as in \cite{SJP}, we get 

\begin{equation}
    \text{Tr}(f)=f+\ell^{1-\frac{k}{2}}(f\mid_k W_\ell)\mid U(\ell).
\end{equation}
Note that Tr maps $M_k(\Gamma_0(\ell))$ to $M_k$.

It could be shown (for instance, using methods similar to Chapter 4 of \cite{fcmf}) that 
\[
F_{k,\chi}|_kW_\ell=\frac{\frakg_\chi B_{k,\chi}}{-2k}\ell^{-\frac{k}{2}}\chi(-1)E_{k,\chi},
\]

and 
\begin{align}\label{ew}
	E_{k,\chi}|_kW_\ell=\frac{-2k\ell^\frac{k}{2}}{\frakg_\chi B_{k,\chi}}F_{k,\chi}
\end{align}

As we shall see, in the study of modular forms of level $\ell$ and character $\chi_\ell$, the Eisenestein series $E_\chi:=E_{\frac{\ell-1}{2},\chi}$ plays a crucial role, similar to that of $E_{\ell-1}$, in the context of $\ell$-adic congruences. This stems from the fact that the denominator of $B_{\frac{\ell-1}{2},\chi}$ is \emph{exactly} divisible by $\ell$ (see \cite{Carlitz1959} for instance), hence we have 
\be \label{E1}
E_\chi\equiv 1 \pmod{\ell}. 
\ee
Setting $F_\chi:=F_{\frac{\ell-1}{2},\chi}$, we get from \eqref{ew}

\be \label{EW}
E_{\chi}|_\frac{\ell-1}{2}W_\ell =C_\ell F_{\chi},
\ee
where

\begin{equation}\label{atkC}
    C_\ell=\begin{cases}
       \frac{-(\ell-1) \ell^\frac{\ell-3}{4}}{B_{\frac{\ell-1}{2},\chi}}\in \mathbb{Q}(\sqrt{\ell})  \textrm{ if $\ell \equiv 1 \pmod 4$},  \\
      \frac{(\ell-1)  \ell^\frac{\ell-3}{4}\sqrt{-1}} {B_{\frac{\ell-1}{2},\chi}}\in \mathbb{Q}(i)  \textrm{ if $\ell \equiv 3 \pmod 4$}.
    \end{cases} 
\end{equation}
Thus, unlike the cases with trivial Nebentypus addressed in \cite{FKO} and \cite{BelmontLee}, we need to work in the ring of integers of $K_\ell:=\Q(C_\ell)$ in order to study ``congruences" modulo $\ell$. More precisely, let $\mathfrak{l}$ denote a prime ideal in the ring of integers of $K_\ell$ above $\ell$, and $v_\mathfrak{l}$ denote the valuation with respect to $\mathfrak{l}$, normalized so that $v_\mathfrak{l}(\mathfrak{l})=1$. We thus have
\begin{equation}\label{v}
   v_\mathfrak{l}( \ell)=\begin{cases}
       2  \textrm{ if $\ell \equiv 1 \pmod 4$},  \\
      1  \textrm{ if $\ell \equiv 3 \pmod 4$}.
    \end{cases} 
\end{equation}
In what follows we shall drop the subscript $\mathfrak{l}$ on $v$ whenever the context it clear. We also extend the definition of the valuation $v$ to Fourier series with rational coefficients in the standard way by setting 
\be\label{vserdef}
v\left(\sum_{n=n_0}^\infty a_n q^n\right):=\inf\{v(a_n):n\geq n_0\}.
\ee
Applying $v$ to \eqref{atkC} we get

\begin{equation}
 \label{Evalution}   
v\left(E_\chi|_\frac{\ell-1}{2}W_\ell\right)=v(C_\ell)=\frac{\ell+1}{4}v(\ell).
\end{equation}

\begin{lemma}\label{hchi}
Set    
\be
h_\chi:=E_\chi-C_\ell E_\chi|_\frac{\ell-1}{2}W_\ell=E_\chi-C_\ell^2 F_\chi \in \mathbb{Z}_{(\ell)}\llbracket q \rrbracket.
\ee
Then we have:

\begin{enumerate}
\item $h_\chi\equiv 1\pmod \ell$.
\item $
h_\chi|_\frac{(\ell-1)}{2}W_\ell=C_\ell F_\chi- \chi(-1)C_\ell E_\chi=C_\ell (F_\chi-\chi(-1)E_\chi).
$
\item $v\left(h_\chi|_\frac{(\ell-1)}{2}W_\ell\right)=\frac{(\ell+1)}{4}v(\ell)$.
\end{enumerate}
\end{lemma}
\begin{proof}
    The proof follows easily from \eqref{E1}, \eqref{EW} and \eqref{Evalution}.
\end{proof}

Thus, in this setting, we have the following generalization of Lemme 9 of \cite{SJP}.

\begin{lemma}
\label{Valuation}
For $f\in M_\kappa(\Gamma_0(\ell),\chi_\ell)$ set $f_m:={\text Tr}(fh_\chi^{\ell^m})$, then we have
\begin{equation}\label{equ4.1}
v(f_m-f)\geq \min(v(f)+(m+1)v(\ell),v(f|_\kappa W_\ell)+\frac{v(\ell)}{2}(\ell^m-\kappa+2))
\end{equation}
\end{lemma}

\begin{proof}
Let $k_m$ denote the weight of $f_m$, namely $k_m=\kappa + \frac{(\ell-1)\ell^m}{2}$. We have 

\begin{align*}
f_m-f&=fh_\chi^{\ell^m}+\ell^{1-\frac{k_m}{2}}fh_\chi^{\ell^m}|_{k_m}W_\ell|U(\ell)-f\\
&=f(h_\chi^{\ell^m}-1)+\ell^{1-\frac{k_m}{2}}fh_\chi^{\ell^m}|_{k_m}W_\ell|U(\ell).
\end{align*}
So, 

\[
v(f_m-f)\geq \min(v(f(h_\chi^{\ell^m}-1)),v(\ell^{1-\frac{k_m}{2}}fh_\chi^{\ell^m}|_{k_m}W_\ell|U(\ell))).
\]
By  Lemma \ref{hchi} (1) and Lemma \ref{lemmaCon1}, the first part gives
\[
v(f(h_\chi^{\ell^m}-1))=v(f)+v(h_\chi^{\ell^m}-1)\geq v(f)+(m+1)v(\ell),
\]
while the second part leads to
\begin{align*}
v(\ell^{1-\frac{k_m}{2}}fh_\chi^{\ell^m}|_{k_m}W_\ell|U(\ell))&\geq v(\ell^{1-\frac{k_m}{2}}fh_\chi^{\ell^m}|_{k_m}W_\ell))
=\left(1-\frac{k_m}{2}\right)v(\ell)+v(f|_\kappa W_\ell)+\ell^m\frac{(\ell+1)}{4}v(\ell)\\
&=v(f|_\kappa W_\ell)+v(\ell)\left(1+\frac{\ell^m (\ell+1)}{4}-\frac{\kappa}{2}-\frac{\ell^m(\ell-1)}{4} \right)\\
&=v(f|_\kappa W_\ell)+\frac{v(\ell)}{2}(\ell^m-\kappa+2),
\end{align*}
(where the first inequality follows since $v(F|U(\ell))\geq v(F)$ for any $F\in \Z_{(\ell)}\llbracket q\rrbracket$, as the coefficients of $F|U(\ell)$ are a subset of those of $F$) and the result follows.
\end{proof}


\section{\bf $\ell$-adic properties of $U(\ell)$ and $D_c(\ell)$ and higher congruences for $R_\ell(b;z)$}
In this section, we show that the functions $R_\ell(b;z)$, which are weakly holomorphic modular forms with Nebentypus $\chi_\ell=\left(\frac{\bullet}{\ell}\right)$, are congruent modulo $\ell$  to a cusp form of level $1$. 
Following \cite{FKO}, we define

\begin{equation}
    A_\ell(z):=\frac{\eta^\ell(z)}{\eta(\ell z)}.
\end{equation}
Using Lemma \ref{etaquo}, we see that $A_\ell(z)\in M_\frac{\ell-1}{2}(\Gamma_0(\ell),\chi_\ell)$. Also, for $ m\geq 1$ we have 

\begin{equation}\label{amodlm}
    A_\ell(z)^{\ell^{m-1}}\equiv 1 \pmod{\ell^m}.
    \end{equation}


\begin{proposition}
\label{Rl0congru}
   For $m\geq 1$,  $R_\ell(0;z)$ is congruent modulo $\ell^m$ to a cusp form in $ S_{\frac{1}{2}\ell^{m-1}(\ell-1)+\frac{c}{2}}\cap  \mathbb{Z}\llbracket q \rrbracket$.

\end{proposition}

\begin{proof} 

For $m=1$, we have 
\begin{align*}
        R_\ell(0;z)&= \eta(\ell z)\eta(z)^{c-1}\\&\equiv\eta(z)^\ell\eta(z)^{c-1}\equiv\Delta(z)^{\frac{\ell+c-1}{24}}\equiv \Delta(z)^{\lceil\frac{\ell-1}{24}\rceil} \pmod \ell.
\end{align*}
  
For $m\geq 2$, we see from Lemma \ref{etaquo} that $
R_\ell(0;z)A_\ell(z)^{\ell^{m-1}}\in M_{\frac{1}{2}\ell^{m-1}(\ell-1)+\frac{c}{2}}(\Gamma_0(\ell))\cap \Z\llbracket q\rrbracket$,       
and $\text{Tr}\left( R_\ell(0;z)A_\ell(z)^{\ell^{m-1}}\right)$ equals

\begin{equation}\label{trr0}
     R_\ell(0;z)A_\ell(z)^{\ell^{m-1}}+\ell^{1-\frac{{\ell^{m-1}(\ell-1)+c}}{4}}\times \left( R_\ell(0;z)\mid_{\frac{c}{2}}W_\ell \cdot A_\ell(z)^{\ell^{m-1}}\mid_{\frac{\ell^{m-1}(\ell-1)}{2}}W_\ell \right)\mid U(\ell) .
\end{equation}

The first term is congruent to $ R_\ell(0;z)$ modulo $\ell^m$ by \eqref{amodlm}, so it suffices to show the second term has valuation $v$ at least $mv(\ell)$.
Standard calculations using the transformation properties of $\eta(z)$ give:

\[
A_\ell(z)^{\ell^{m-1}}|_{\ell^{m-1}(\ell-1)/2}W_\ell=\ell^{\frac{\ell^{m}+\ell^{m-1}}{4}}(-i)^{\ell^{m-1}(\ell-1)/2}\left(\frac{\eta\left(\ell z\right)^\ell}{\eta\left( z\right)}\right)^{\ell^{m-1}}, 
\]

\begin{align*}
       R_\ell(0;z)|_{\frac{c}{2}}W_\ell&= \ell^{\frac{c}{4}}(\ell z)^{-c/2}  \cdot \eta\left(\frac{-1}{z}\right)\eta\left(\frac{-1}{\ell z}\right)^{c-1}\\
       &= \ell^{\frac{c}{4}-\frac{1}{2}}  \cdot(-i)^{c/2} \eta(z)\eta(\ell z)^{c-1}.
\end{align*}
Thus the second term of \eqref{trr0} is equal to $\ell^{\frac{1}{2}\ell^{m-1}+\frac{1}{2}}S,$ where 
\[
S:=(-i)^\frac{\ell^{m-1}(\ell-1)+c}{2}\left(\frac{\eta(\ell z)^{\ell}}{\eta(z)}\right)^{\ell^{m-1}} \eta(z)\eta(\ell z)^{c-1}|U(\ell)
\]
is a $q$-series with integer coefficients (note that the exponent on $(-i)$ is always even). It follows that  
\[
v(\ell^{\frac{1}{2}\ell^{m-1}+\frac{1}{2}}S)\geq \frac{1}{2}(\ell^{m-1}+1)v(\ell)\geq mv(\ell).
\]
Hence

\begin{equation*}
        \ell^{1-\frac{{\ell^{m-1}(\ell-1)+c}}{4}}\times \left( R_\ell(0;z)\mid_{\frac{c}{2}}W_\ell \cdot A_\ell(z)^{\ell^{m-1}}\mid_{\frac{\ell^{m-1}(\ell-1)}{2}}W_\ell \right)\mid U(\ell)\equiv0\pmod{\ell^m}.
\end{equation*}

\end{proof}


\begin{lemma}
\label{baseseclem}
Let $\ell \geq 5$ be prime, and suppose that $\Psi(z)\in \mathbb{Z}_{(\ell)}\llbracket q \rrbracket$ such that for all $ 1\leq j\leq m$ there exists $g_j(z)\in M_{\frac{1}{2}\ell^{j-1}(\ell-1
)+\frac{c}{2}}\cap \mathbb{Z}_{(\ell)}\llbracket q \rrbracket$ with $\Psi(z)\equiv g_j(z)\pmod{\ell^j}$. Set $\calA:=\{(5,3),(7,2),(11,2),(13,2)\}$. Then for all $ 1\leq j\leq m$, there exists $h_j(z) \in 
 S_{ k_\ell(j)}\cap\mathbb{Z}_{(\ell)}\llbracket q \rrbracket$ such that $\Psi(z)\mid D_c(\ell)\equiv h_j(z)\pmod{\ell^j}$
where 

\begin{equation}\label{klj}
    k_\ell(j)=\begin{cases}
        \frac{1}{2}(c\ell+\ell-1)  \textrm{ if $j=1$}, \\
        \frac{1}{2} \ell^{j-1}(\ell-1)+\frac{c}{2} \textrm{ if $j\geq 2$ and $(\ell,j)\nin \calA\cup\{(5,2)\}$},\\
        \frac{1}{2} \ell^{j}(\ell-1)+\frac{c}{2} \textrm{ if $(\ell,j)\in \calA$},\\
        k_5(4)=260 \textrm{ if $(\ell,j)=(5,2)$}.
    \end{cases} 
\end{equation}
\end{lemma}

\begin{proof} 
Note that $k_\ell(j)\equiv \frac{c+\ell-1}{2}\equiv 12 \left\lceil \frac{\ell-1}{24}\right\rceil$ so \eqref{cong1} is satisfied. Thus the base case $j=1$ follows by Lemma \ref{baselem}(b). For $j>1$ we have the following congruence for $\Psi(z)|D_c(\ell)$. 
 
\begin{equation}
\label{4parts}
    \begin{split}
    \Psi(z)\mid D_c(\ell)&\equiv g_j(z)\mid D_c(\ell)\equiv\left(g_j(z)-g_{j-1}E_{\ell-1}(z)^{\frac{1}{2}\ell^{j-2}(\ell-1)}\right)\mid D_c(\ell)\\&+\left( g_{j-1}E_{\ell-1}(z)^{\frac{1}{2}\ell^{j-2}(\ell-1)}-g_{j-1}(z)E_{\ell-1}(z)^{\frac{1}{2}\ell^{j-2}(\ell-1)}\frac{E_{\chi}(z)^{\ell^{j-2}}}{A_{\ell}(z)^{\ell^{j-2}}}\right)\mid D_c(\ell) \\& +\left(g_{j-1}(z)E_{\ell-1}(z)^{\frac{1}{2}\ell^{j-2}(\ell-1)}\frac{E_{\chi}(z)^{\ell^{j-2}}}{A_{\ell}(z)^{\ell^{j-2}}}-g_{j-1}(z)\frac{E_{\chi}(z)^{\ell^{j-1}}}{A_{\ell}(z)^{\ell^{j-2}}}\right)\mid D_c(\ell)\\&+\left( \frac{g_{j-1}(z)}{A_{\ell}(z)^{\ell^{j-2}}}E_{\chi}(z)^{\ell^{j-1}}\right) \mid D_c(z)\pmod{\ell^j}.
    \end{split}
\end{equation}


Since $E_{\chi}(z)^{\ell^{j-1}}\equiv 1 \pmod{\ell^j}$, $g_{j-1}(z)\in M_{\frac{1}{2}\ell^{j-2}(\ell-1)+\frac{c}{2}}$ and $A_{\ell}(z)^{\ell^{j-2}}\in M_{\frac{1}{2}\ell^{j-2}(\ell-1)}(\Gamma_0(\ell),\chi_\ell)$, then the fourth summand  simplifies $\pmod{\ell^j}$ to 



\begin{equation}
\label{main1}
   \mathcal{F}_j(z):=\left( \frac{g_{j-1}(z)}{A_{\ell}(z)^{\ell^{j-2}}}\right)\mid D_c(\ell)\in M_{\frac{c}{2}}(\Gamma_0(\ell),\chi_\ell).
\end{equation}
The following proposition addresses $\calF_j \pmod{\ell^j}$.

\begin{proposition}
\label{Fj} For prime $\ell\geq 5$ and integer $j\geq 2$,
    the form $\mathcal{F}_j(z)$ is congruent modulo $\ell^j$ to a form in $S_{\frac{1}{2}\ell^{j-1}(\ell-1)+\frac{c}{2}}\cap\mathbb{Z}_{(\ell)}\llbracket q \rrbracket$, except possibly for $(\ell,j)\in \{(5,2), (5,3), (7,2), (11,2), (13,2)\}$.
\end{proposition}

\begin{remark}
 The exceptional cases of Proposition \ref{Fj} are what lead to the higher weights $k_\ell(j)$ for $(\ell,j) \in \calA\cup\{(5,2)\}$ in \eqref{klj}, which are simply the weights $k_\ell(j+1)$ or $k_5(4)$ where the congruences hold modulo $\ell^{j+1}$ or $5^4$ and hence modulo $\ell^j$.
 However, often times there are intermediate weights less than  $k_\ell(j+1)$ where the congruence modulo $\ell^j$ will be satisfied. Our main interest is in the forms $R_\ell(b;z)$ and we see that, for example 

\be\label{53ex} 
R_{5}(1;z)=R_{5}(0;z)\mid D_{20}(5) =\frac{\eta(5^2z)^{20}\eta(5z)}{\eta(z)}\mid U(5)\equiv 5\Delta(z)^5\pmod{5^3}.
\ee
The congruence \eqref{53ex} can also be viewed mod $5^2$, and it takes place in weight $60$, which actually is less than $k_5(4)=260$ guaranteed in the Proposition for a general form.  For the other cases we also have

\begin{align*}
  R_{7}(1;z)&\equiv 7\Delta(z)^6\pmod{7^2},\\
  R_{11}(1;z)&\equiv 11\Delta(z)^7E_{88}(z) \pmod{11^2},\\
R_{13}(1;z)&\equiv 11\Delta^7E_{156}+130\Delta^8 E_{144} \pmod{13^2}.
  \end{align*}
 
\end{remark}

\begin{proof}[Proof of Proposition \ref{Fj}]
   We have 

\begin{equation*}
       \mathcal{F}_j(z)\equiv \mathcal{F}_j(z)h_\chi(z)^{\ell^{j-1}}\pmod{\ell^j}.
\end{equation*}
The form on the right hand side is in $M_{\frac{1}{2}\ell^{j-1}(\ell-1)+\frac{c}{2}}^!(\Gamma_0(\ell))\cap\mathbb{Z}_{(\ell)}
\llbracket q\rrbracket$. Since $\Tr(M_k^!(\Gamma_0(\ell)))\subset M_k^!$, we see that $\Tr\left(\mathcal{F}_j(z)h_{\chi}(z)^{\ell^{j-1}}\right)$ is on SL$_2(z)$ with weight $\frac{1}{2}\ell^{j-1}(\ell-1)+\frac{c}{2}$. Hence it suffices to show that 

\begin{equation*}
    \mathcal{F}_j(z)\equiv\text{Tr}\left(\mathcal{F}_j(z)h_{\chi}(z)^{\ell^{j-1}}\right)\pmod{\ell^j}.
\end{equation*} 

We apply Lemma \ref{Valuation} with $m=j-1$ and $f=\calF_j=g_{j-1}A_\ell^{-\ell^{j-2}}\Phi_\ell^c\mid U(\ell)$, which is of weight $\kappa=\frac{c}{2}$. Since $f$ has $\ell$-integral coefficients, we have $v(f)\geq 0$ and the first part of \eqref{equ4.1} is $\geq j v(\ell)$. For the second part, standard computations as in the proof of Proposition 3.2 in \cite{BMJ} give




\begin{align*}
    g_{j-1}&\mid_{\frac{1}{2}\ell^{j-2}(\ell-1)+\frac{c}{2}}\left(\left(\begin{array}{cc}
    1&0  \\
     0&\ell 
\end{array}\right)\left(\begin{array}{cc}
    0 &-1  \\
     \ell&0
\end{array}\right)\right)=\ell^{\frac{1}{2}\ell^{j-2}(\ell-1)+\frac{c}{2}}g_{j-1}(\ell^2z),
\end{align*}

\begin{align*}
    A_\ell(z)^{\ell^{j-2}}&\mid_{\frac{1}{2}\ell^{j-2}(\ell-1)}\left(\left(\begin{array}{cc}
    1&0  \\
     0&\ell 
\end{array}\right)\left(\begin{array}{cc}
    0 &-1  \\
     \ell&0
\end{array}\right)\right)=\ell^{\frac{1}{2}\ell^{j-1}}i^{-\frac{1}{2}\ell^{j-2}(\ell-1)}\left(\frac{\eta\left(\ell^2z\right)^\ell}{\eta\left(\ell z\right)}\right)^{\ell^{j-2}},
\end{align*}

\begin{align*}
    \Phi_\ell^c(z)&\mid_0\left( \left(\begin{array}{cc}
    1 &0  \\
     0&\ell 
\end{array}\right)\left(\begin{array}{cc}
    0 &-1  \\
     \ell&0 
\end{array}\right)\right)=\frac{1}{\ell^c\Phi_\ell^c(z)}.
\end{align*}
Hence
\begin{equation} 
\label{fW}
v(\mathcal{F}_j|_\kappa W_\ell)=v(\ell)\left(\frac{\ell^{j-2}(\ell-1)+c}{2}-\ell^{j-2}\frac{\ell}{2}-c\right)=\frac{-v(\ell)}{2}(\ell^{j-2}+c).
\end{equation}

Thus, we finally have

\begin{equation}\label{val-diff-trace}
v\left(\text{Tr}\left(g_{j-1}(z)A_\ell^{-\ell^{j-2}}|D_c(\ell)\right)-g_{j-1}(z)A_\ell^{-\ell^{j-2}}|D_c(\ell)\right)\geq \frac{v(\ell)}{2}(\ell^{j-2}(\ell-1)-\frac{3c}{2}+2).
\end{equation}

In order for $g_{j-1}(z)A_\ell^{-\ell^{j-2}}|D_c(\ell)$ to be congruent to its trace modulo $\ell^j$, we need to have the valuation in \eqref{val-diff-trace} be $\geq jv(\ell)$.
A straightforward verification shows that the inequality is indeed satisfied for all $(\ell, j)$ except $(5,2), (5,3), (7,2), (11,2), (13,2)$.

\end{proof}


To study the second summand of \eqref{4parts} set



\begin{equation*}
    B_{j,\ell}(z):= g_{j-1}-g_{j-1}(z)\frac{E_{\chi}(z)^{\ell^{j-2}}}{A_{\ell}(z)^{\ell^{j-2}}}\equiv 0\pmod{\ell^{j-1}}.
\end{equation*}
We also have $E_{\ell-1}(z)^{\frac{1}{2}\ell^{j-2}(\ell-1)}-1\equiv 0 \pmod{\ell^{j-1}}$, hence 



\begin{equation*}
    E_{\ell-1}(z)^{\frac{1}{2}\ell^{j-2}(\ell-1)}B_{j,\ell}(z)\equiv B_{j,\ell}(z)\pmod{\ell^{2j-2}}.
\end{equation*}
Since $j\geq 2$ we have $2j-2\geq j$. It follows that the second summand of \eqref{4parts} modulo $\ell^j$ is 

\begin{equation}
    \label{Sum22}
    B_{j,\ell}(z)|D_c(\ell)\in M_{\frac{1}{2}\ell^{j-2}(\ell-1)+\frac{c}{2}}(\Gamma_0(\ell))\cap\mathbb{Z}_{(\ell)}\llbracket q \rrbracket
\end{equation}

\begin{proposition}
\label{pro2}
 The form  $B_{j,\ell}(z)|D_c(\ell)$ is congruent modulo $\ell^j$ to a form in $S_{\frac{1}{2}\ell^{j-1}(\ell-1)+\frac{c}{2}}\cap\mathbb{Z}_{(\ell)}\llbracket q \rrbracket$.

\end{proposition}

\begin{proof} 
We have
\begin{align*}
\dfrac{B_{j,\ell}(z)}{\ell^{j-1}}\mid D_c(\ell)&=\left(\dfrac{B_{j,\ell}(z)}{\ell^{j-1}}\cdot\Phi_\ell^c(z)\right)\mid U(\ell)\equiv \left(\frac{B_{j,\ell}(z)}{\ell^{j-1}}~\Delta(z)^{\frac{c(\ell^2-1)}{24}}\right)\mid U(\ell)\\&\equiv  \left(\dfrac{B_{j,\ell}(z)}{\ell^{j-1}}~\Delta(z)^{\frac{c(\ell^2-1)}{24}}\right)\mid U(\ell)  E_{\ell-1}(z)^{\frac{1}{2}\ell^{j-2}(\ell-1)-c\left(\frac{\ell+1}{2}\right)}\pmod{\ell}.
\end{align*} 
Multiplying by $\ell^{j-1}$ gives

\begin{align*}
    B_{j,\ell}(z)\mid D_c(\ell)\equiv  \left(B_{j,\ell}(z)~\Delta(z)^{\frac{c(\ell^2-1)}{24}}\right)\mid U(\ell)  E_{\ell-1}(z)^{\frac{1}{2}\ell^{j-2}(\ell-1)-c\left(\frac{\ell+1}{2}\right)}\pmod{\ell^j}.
\end{align*}
This form lies in $M_{\frac{1}{2}\ell^{j-1}(\ell-1)+\frac{c}{2}}^!(\Gamma_0(\ell))\cap\mathbb{Z}_{(\ell)}\llbracket q \rrbracket$. It remains to show that 

\begin{align*}
   \left(B_{j,\ell}(z)~\Delta(z)^{\frac{c(\ell^2-1)}{24}}\right)\mid U(\ell)=\left(\left( g_{j-1}-g_{j-1}(z)\frac{E_{\chi}(z)^{\ell^{j-2}}}{A_{\ell}(z)^{\ell^{j-2}}}\right)~\Delta(z)^{\frac{c(\ell^2-1)}{24}}\right)\mid U(\ell)
\end{align*}
is congruent modulo $\ell^j$ to a cusp form on SL$_2(\mathbb{Z})$. Since $g_{j-1}(z)\Delta(z)^{\frac{c(\ell^2-1)}{24}}\in S_{\frac{1}{2}\ell^{j-2}(\ell-1)+c\left(\frac{\ell^2-1}{2}\right)+\frac{c}{2}}$ and $\frac{1}{2}\ell^{j-2}(\ell-1)+c\left(\frac{\ell^2-1}{2}\right)+\frac{c}{2}-1\geq j$, we see from Proposition \ref{Stab} that $g_{j-1}(z)\Delta(z)^{\frac{c(\ell^2-1)}{24}}|U(\ell)$ is congruent modulo $\ell^j$ to a cusp form in the same space. Therefore it suffices to show that 

\begin{equation*}
   \left(g_{j-1}(z)\frac{E_{\chi}(z)^{\ell^{j-2}}}{A_{\ell}(z)^{\ell^{j-2}}}~\Delta(z)^{\frac{c(\ell^2-1)}{24}}\right)\mid U(\ell)
\end{equation*}
is congruent to a form in $S_{\frac{1}{2}\ell^{j-2}(\ell-1)+c\left(\frac{\ell^2-1}{2}\right)+\frac{c}{2}}$. For convenience, we define 

\begin{equation*}
    C_{j,\ell}(z):=g_{j-1}(z)\frac{E_{\chi}(z)^{\ell^{j-2}}}{A_{\ell}(z)^{\ell^{j-2}}}~\Delta(z)^{\frac{c(\ell^2-1)}{24}}\in M_{\frac{1}{2}\ell^{j-2}(\ell-1)+c\left(\frac{\ell^2-1}{2}\right)+\frac{c}{2}}^!(\Gamma_0(\ell)).
\end{equation*}
Applying the trace operator to $C_{j,\ell}|W_\ell$ we obtain

\begin{align*}
    \ell^{\frac{\ell^{j-2}(\ell-1)+c\left(\ell^2-1\right)+c}{4}-1}&\text{Tr}\left(C_{j,\ell}(z)\mid_{\frac{1}{2}\ell^{j-2}(\ell-1)+c\left(\frac{\ell^2-1}{2}\right)+\frac{c}{2}}W_\ell\right)\\&=C_{j,\ell}(z)\mid U(\ell)+\ell^{\frac{\ell^{j-2}(\ell-1)+c\left(\ell^2-1\right)+c}{4}-1}C_{j,\ell}(z)\mid_{\frac{1}{2}\ell^{j-2}(\ell-1)+c\left(\frac{\ell^2-1}{2}\right)+\frac{c}{2}}W_\ell.
\end{align*}
This form is on SL$_2(\mathbb{Z})$ by Lemma 2.1(3) in \cite{BMJ}. We shall show that 

\begin{equation*}
   \ell^{\frac{\ell^{j-2}(\ell-1)+c\left(\ell^2-1\right)+c}{4}-1}C_{j,\ell}(z)\mid_{\frac{1}{2}\ell^{j-2}(\ell-1)+c\left(\frac{\ell^2-1}{2}\right)+\frac{c}{2}}W_\ell\equiv 0 \pmod{\ell^j}.
\end{equation*}
Following the proof of Proposition 3.3 in \cite{BMJ}, we obtain

\begin{align*}
   g_{j-1}(z)\Delta(z)^{\frac{c(\ell^2-1)}{24}}\mid_{\frac{1}{2}\ell^{j-2}(\ell-1)+\frac{c}{2}(\ell^2-1)+\frac{c}{2}}W_\ell=\ell^{\frac{1}{4}\ell^{j-2}(\ell-1)+\frac{c}{4}(\ell^2-1)+\frac{c}{4}}g_{j-1}(\ell z)\Delta(\ell z)^{\frac{c(\ell^2-1)}{24}}.
\end{align*}
From \eqref{Evalution} and direct computation with the transformation law for $\eta(z)$, we have
\begin{align*}
  & v\left(E_\chi(z)^{\ell^{j-2}}\mid_{\frac{1}{2}\ell^{j-2}(\ell-1)} W_\ell\right)=\ell^{j-2}v\left(E_\chi\mid _\frac{\ell-1}{2}W_\ell\right)=\frac{\ell^{j-2}(\ell+1)}{4}v(\ell),\\&
    v\left(A_\ell(z)^{-\ell^{j-2}}\mid_{\frac{1}{2}\ell^{j-2}(\ell-1)} W_\ell\right)=-\frac{1}{4}(\ell^{j-1}+\ell^{j-2})v(\ell).
\end{align*}
We thus get
\begin{align*}
    &v\left( \ell^{\frac{\ell^{j-2}(\ell-1)+c\left(\ell^2-1\right)+c}{4}-1}C_{j,\ell}(z)\mid_{\frac{1}{2}\ell^{j-2}(\ell-1)+c\left(\frac{\ell^2-1}{2}\right)+\frac{c}{2}}W_\ell\right)
    =\left(\frac{1}{2}\ell^{j-1}+\frac{1}{2}c\ell^2-\frac{1}{2}\ell^{j-2}-1\right)v(\ell)\geq jv(\ell),
\end{align*}
completing the proof of the Proposition.

\end{proof}

In a similar manner to the above (and also to the proof of Lemma 3.1 in \cite{BMJ}), we can prove that each of the first and third summands of \eqref{4parts} is congruent modulo $\ell^j$ to a form in $S_{\frac{1}{2}(\ell^{j-1}(\ell-1)+\frac{c}{2}}\cap \Z_{(\ell)}\llbracket q\rrbracket$, and we omit the details for brevity. This completes the proof of Lemma \ref{baseseclem}.  
\end{proof}

\begin{lemma}
\label{cuspodev}
If $b\geq 0$ is even, then $R_\ell(b;z)$ is congruent modulo $\ell^m$ to a form in $S_{\frac{\ell^{m-1}(\ell-1)+c}{2}}$. If $b$ is odd, then $R_\ell(b;z)$ is congruent modulo $\ell^m$ to a form in  $S_{k_\ell(m)}$, where $k_\ell(m)$ is defined by \eqref{klj}.
\end{lemma}

\begin{proof}
The base case $b=0$ follows directly from Proposition \ref{Rl0congru}. For the induction step, let $b\geq0$ be an even integer for which the statement holds.We can thus apply Lemma \ref{baseseclem} to deduce that there exists $h_m(b;z)\in S_{k_\ell(m)}\cap\mathbb{Z}_{(\ell)}\llbracket q \rrbracket$ with   
\[
R_\ell(b+1;z)=R_\ell(b;z)\mid D_c(\ell)\equiv h_m(b,z)\pmod{\ell^m}.
\]
For $m=1$ and all primes $\ell\geq5$, by Lemma \ref{baselem} there exists $h_1(b+1;z)\in S_{\frac{1}{2}(\ell-1)+\frac{c}{2}}\cap\mathbb{Z}_{(\ell)}\llbracket q \rrbracket$ with
\[
R_\ell(b+2;z)=R_\ell(b+1;z)\mid U(\ell)\equiv h_1(b+1;z)\pmod{\ell}.
\]
For $m\geq2$ and primes $\ell\geq5$ we have 
\[k_\ell(m)\geq \frac{1}{2}\ell^{m-1}(\ell-1
)+\frac{c}{2}\geq m.
\]
It follows from Proposition \ref{Stab} that 
\[
R_\ell(b+2;z)=R_\ell(b+1;z)\mid U(\ell)\equiv h_m(b;z)\mid U(\ell)\equiv h_m(b;z)\mid T(\ell,k_\ell(m))\pmod{\ell^m},
\]
and the result follows since  the Hecke operator $ T(\ell,k_\ell(m))$ preserves spaces of level one cusp forms.

\end{proof}

\begin{example}
    For $\ell=17$ we have $c=8$. 
    The following explicit congruences, which were computed using PARI/GP \cite{PARI2}, illustrate Lemma \ref{cuspodev} for $R_{17}(b;z) \pmod{17^j}$ for $0\leq b \leq 3$ and $1\leq j \leq 3$.
\begin{align*}
     R_{17}(0;z)&\equiv \Delta(z)\pmod{17},\\
     R_{17}(0;z)&\equiv \Delta(z)E_{128}(z)\pmod{17^2},\\ 
R_{17}(0;z)&\equiv \Delta(z)E_{2304}(z)+ 3179\Delta^2(z)E_{2292}(z)+867\Delta^3(z)E_{2280}(z) \pmod{17^3}. 
\end{align*} 
\begin{align*}
     R_{17}(1;z)&\equiv 7\Delta^6(z)E_4(z) \pmod{17},\\
     R_{17}(1;z)&\equiv 7\Delta^6(z)E_{68}(z)+ 34\Delta^7(z)E_{56}(z) \pmod{17^2},\\
 R_{17}(1;z)&\equiv 7\Delta^6(z)E_{2244}(z)+ 2346\Delta^7(z)E_{2232}(z)+3757\Delta^8(z)E_{2220}(z)\\&+ 3468\Delta^9(z)E_{2208}(z) \pmod{17^3}.
\end{align*} 
\begin{align*}
     R_{17}(2;z)&\equiv 11\Delta(z) \pmod{17},\\
 R_{17}(2;z)&\equiv 96\Delta(z)E_{128}(z)+ 34\Delta^2(z)E_{114}(z) \pmod{17^2},\\ 
 R_{17}(2;z)&\equiv 1252\Delta(z)E_{2304}(z)+ 612\Delta^2(z)E_{2292}(z)+867\Delta^3(z)E_{2280}(z) \pmod{17^3}. 
\end{align*}
 \begin{align*}
     R_{17}(3;z)&\equiv 9\Delta^6(z)E_4(z) \pmod{17},\\
     R_{17}(3;z)&\equiv 94\Delta^6(z)E_{68}(z)+ 85\Delta^7(z)E_{56}(z) \pmod{17^2},\\
      R_{17}(3;z)&\equiv 2117\Delta^6(z)E_{2244}(z)+ 4131\Delta^7(z)E_{2232}(z)+2023\Delta^8(z)E_{2220}(z)\\&+ 3757\Delta^9(z)E_{2208}(z) \pmod{17^3}.
 \end{align*}
Note that, for example, $R_{17}(3;z)\equiv 2408 R_{17}(1;z) \pmod{17^3}$. This is consistent with Theorem \ref{thm1}, which we shall attend to proving in the following section. 
    
 \end{example}


\section{An injection of $\Omega_{\ell,\text{reg}}^\text{odd}(m)$ to cusp forms and the proof of Theorem \ref{thm1}}

Let $S^{\text{even}}(\ell,m)$ and $S^{\text{odd}}(\ell,m)$ denote the $\Z/\ell^m\Z$-modules obtained by reducing modulo $\ell^m$ the forms in $S_{\frac{1}{2}\ell^{m-1}(\ell-1)+\frac{c}{2}}\cap\mathbb{Z}_{(\ell)}\llbracket q \rrbracket$ and $S_{k_\ell(m)}\cap\mathbb{Z}_{(\ell)}\llbracket q \rrbracket$, respectively. Note that both modules are finite since each space of modular forms is generated by finitely many modular forms with integer coefficients. 

\begin{lemma}\label{xcisom}
    The operator $X_c(\ell)$ is an automorphism of the $\Z/\ell^m \Z$-module $\Omega_{\ell,\text{reg}}^{\text{odd}}(m)$. Furthermore, there exists a positive integer $n$ such that $X_c^n(\ell)$ is the identity map on $\Omega_{\ell,\text{reg}}^{\text{odd}}(m)$. 
\end{lemma}
\begin{proof}
From Lemma \ref{cuspodev}, we have the following nesting of $\mathbb{Z}/\ell^m\mathbb{Z}$-modules: 
\begin{align*}
  S^{\text{odd}}(\ell,m)\supseteq \lamrodd(1,m)\supseteq \lamrodd(3,m)\supseteq \cdots \supseteq \lamrodd(2b+1,m)\supseteq \cdots.
\end{align*}
Since $S_{k_\ell(m)}\cap\mathbb{Z}_{(\ell)}\llbracket q \rrbracket$ is finitely generated, we see that $ S^{\text{odd}}(\ell,m)$ is a finite module, hence we must have $\lamrodd(b,m)=\lamrodd(b+2,m)$ within a finite number of steps, afterwards the sequence must stabilize to $\Omega_{\ell,\text{reg}}^\text{odd}(m)$ since the process of going from $b$ to $b+2$ is identical to the one used to go from $b+2$ to $b+4$. Essentially by definition, we have $\lamrodd(b+2,m)=X_c(\lamrodd(b,m))$. It follows that $X_c(\ell)$ is a surjective endomorphism of the finite $\Z/\ell^m \Z$-module $\Omega_{\ell,\text{reg}}^\text{odd}(m)$, and hence a bijection. Since the automorphism group of a finite ring is itself a finite group; it follows that each element of it has a finite order.
\end{proof}


Consider the following submodules of $S^{\text{odd}}(\ell,m)$:
 \[
 S_0=\left\{ f(z)E_{\ell-1}^{\frac{1}{2}(\ell^{m-1}-c-1)}(z): f(z)\in S_{\frac{c\ell+\ell-1}{2}}\cap\mathbb{Z}_{(\ell)}\llbracket q\rrbracket \right\},
 \]
    
\[S_1=\left\{ g(z):g(z)=\sum_{j=m_0}^{\infty}a_jq^j\in S_{\frac{1}{2}\ell^{m-1}(\ell-1)+\frac{c}{2}} \text{ with } m_0> \text{dim}(S_{\frac{c\ell+\ell-1}{2}}) \right\}.
\]
It is well known that any level 1 modular form can be expressed as a polynomial in the forms $E_4,E_6$ which have integral coefficients. Hence we can construct a basis $\{f_i=q^i+O(q^{i+1}): 1\leq i\leq n=\dim(S_{\frac{1}{2}\ell^{m-1}(\ell-1)+\frac{c}{2}})\}$ for $S_{\frac{1}{2}\ell^{m-1}(\ell-1)+\frac{c}{2}}\cap\mathbb{Z}\llbracket q\rrbracket$ with $f_k(z)\in S_0$ for $k\leq \lfloor\frac{c\ell+\ell-1}{24}\rfloor$ and $f_k(z)\in S_1$ otherwise. It follows that $S^\text{odd}(\ell,m)=S_0\bigoplus S_1$. In particular, we have a natural projection map $\pi_0:S^\text{odd}(\ell,m)\rightarrow S_0$

\begin{lemma}
    \label{ordLemm}
Let $\mathcal{S}$ be a $\Z/\ell^m \Z$ submodule of $S^{\text{odd}}(\ell,m)$ on which $X_c(\ell)$ is an isomorphism. Suppose that $f(z)\in \mathcal{S}$ has $v(f)=i<m$, and  $f(z)=f_0(z)+f_1(z)$ with $f_i\in S_i$. Then we have $v(f_1)>i$ and $v(f_0)=i$. Also, as $\Z/\ell^m \Z$ modules we have $\mathcal{S}\cong\pi_0(\mathcal{S})\subset S_0$.
\end{lemma}

\begin{proof}
The proof proceeds as Lemma 4.4 of \cite{BMJ}, and we omit the details for brevity.  
\end{proof}

   




   

\begin{theorem}
\label{injection}
Let $\ell \geq 5$ be prime, and let $m\geq 1$. Then there exist an injective $\mathbb{Z}/\ell^m\mathbb{Z}-$module homomorphism 

\[
    \Pi_o:\Omega_{\ell,\text{reg}}^{\text{odd}}(m)\hookrightarrow S_{\frac{1}{2}(c\ell+\ell-1)}
\]
such that for all $f \in \Omega_{\ell,\text{reg}}^{\text{odd}}(m)$ we have
\begin{enumerate}[label=(\alph*)]
\item $\text{ord}_\infty(f)=\text{ord}_\infty (\Pi_o(f))$, and

    
\item      $\Pi_o(f)\equiv f \pmod{\ell^{v(f)+1}}.$

\end{enumerate}
\end{theorem}

\begin{proof}

Applying Lemma \ref{xcisom} and Lemma \ref{ordLemm}, we see that $\Pi_o(f):={\pi_0(f)}{E_{\ell-1}^{-\frac{1}{2}(\ell^{m-1}-c-1)}}$ indeed gives an injection from 
$\Omega_\ell^{\text{odd}}(m)$ into $S_{\frac{1}{2}(c\ell+\ell-1)}$. Part (a) follows since $\pi_0$ preserves the order of vanishing at $\infty$ and $\text{ord}_\infty (E_{\ell-1})=0$. Furthermore, again by Lemma \ref{ordLemm}, we have $f\equiv \pi_0(f) \pmod{\ell^{i+1}}$ for $f\in \Omega_\ell^{\text{odd}}(m)$ with $v(f)=i$. Utilizing $E_{\ell-1}\equiv 1 \pmod \ell$ we get
\[
\ell^{-i}f\equiv\ell^{-i}\pi_0(f)E_{\ell-1}^{\frac{1}{2}(\ell^{m-1}-c-1)} \pmod \ell.
\]
Multiplying by $\ell^i$ we obtain $f\equiv \Pi_{o}(f) \pmod{\ell^{i+1}}$, completing the proof.

\end{proof}

\begin{proof}[Proof of Theorem \ref{thm1}]

The stabilization of the nested sequence \eqref{nestodd} into $\Omega_{\ell,\text{reg}}^{\text{odd}}(m)$ follows by the proof of Lemma \ref{xcisom}. Since $\text{ord}_\infty(\Phi_\ell^c)=\frac{c(\ell^2-1)}{24}$, it follows that for any $q$-series $F$ with nonnegative exponents we have

\[
    \text{ord}_\infty(F\mid D_c(\ell))\geq \frac{c(\ell^2-1)}{24\ell}.
\]
In particular, since any $f\in \Omega_{\ell,\text{reg}}^{\text{odd}}(m)$ comes from the action of $D_c(\ell)$ on such a $q$-series, we get by Lemma \eqref{ordLemm} (a) that 
\[
\text{ord}_\infty(\Pi_o(f))=\text{ord}_\infty(f)\geq \frac{c(\ell^2-1)}{24 \ell}.
\]
We thus conclude the following bound on the rank:

\begin{equation}\label{rankdim}
    \text{rank}_{\mathbb{Z}/\ell^m\mathbb{Z}}\left(\Omega_{\ell,\text{reg}}^{\text{odd}}(m)\right)\leq \dim(S_{\frac{c\ell+\ell-1}{2}})-\left\lfloor \frac{c(\ell^2-1)}{24\ell}\right\rfloor.
\end{equation}
For $\ell\equiv 1 \pmod{24}$, the bound \eqref{rankbound} follows directly from \eqref{rankdim}. For $\ell\not\equiv 1 \pmod{24}$, write ~$\ell =24t+\delta$ with $5\leq \delta\leq 23$; we thus have $c=25-\delta$. Elementary calculations show that $$\frac{c(\ell^2-1)}{24}=\left(ct+\delta-\frac{\delta^2-1}{24}\right)\ell+\left((\delta-1)t+\frac{\delta^2-1}{24}-1\right),$$ hence we have

\begin{align}\label{rank1}
   \left\lfloor \frac{c(\ell^2-1)}{24\ell}\right\rfloor=ct+\delta-\frac{\delta^2-1}{24}. 
\end{align}
 We note that for all prime $\ell \geq 5$ we have $\frac{c\ell+\ell-1}{2}\not\equiv 2 \pmod {12}$. Utilizing the well-known formula for dimensions of spaces of cusp forms of level $1$ and the elementary formula $$(c+1)(24t+\delta)-1=\left((c+1)t+\delta-\frac{\delta^2-1}{24}\right)24+2(\delta-1),$$ we conclude
\begin{align}\label{rank2}
    \dim(S_{\frac{c\ell+\ell-1}{2}})=\left\lfloor \frac{(c+1)\ell-1}{24}\right\rfloor=(c+1)t+\delta-\frac{\delta^2-1}{24}+\left\lfloor\frac{\delta-1}{12}\right\rfloor.
\end{align}
Combining \eqref{rankdim}, \eqref{rank1} and \eqref{rank2} we obtain \eqref{rankbound} for the remaining cases.

In order to prove the isomorphism $\Omega_{\ell,\text{reg}}^{\text{even}}(m)\cong\Omega_{\ell,\text{reg}}^{\text{odd}}(m)$, note that $\lamrev(b+1,m)$ is the image under $U(\ell)$ of $\lamrodd(b,m)$, hence the finiteness and stabilization follow for the even modules follow from their counterparts for the odd modules. Furthermore, since $X_c(\ell)=U(\ell)\circ D_c(\ell)$, then $U(\ell):\Omega_{\ell,\text{reg}}^{\text{odd}}(m)\rightarrow \Omega_{\ell,\text{reg}}^{\text{odd}}(m)$ must be an isomorphism; completing the proof. 

\end{proof}

\medskip
\begin{remark}
One can similarly construct an injection $\Pi_e:\Omega_{\ell,\text{reg}}^{\text{even}}(m)\hookrightarrow S_{\frac{\ell-1+c}{2}}\cap\mathbb{Z}_{(\ell)}\llbracket q\rrbracket$, leading to a rank bound of $\left\lceil\frac{\ell-1}{24}\right\rceil$. In general, this is a weaker upper bound than \eqref{rankbound}.

\end{remark}

\begin{remark}\label{boundrmrk}
Using the methods of Section 5 of \cite{BMJ}, we can similarly obtain upper bounds for the integers $\mathfrak{b}_\ell(m)$. In the interest of brevity we omit the details but state the results:
\begin{align*}
\begin{split}
& \mathfrak{b}_\ell(1)\leq 2d_\ell +1,\\
&\mathfrak{b}_\ell(m)\leq 2(d_\ell+1)+2(d_\ell'+1)(m-1) \text{ for }~m\geq 2,
\end{split}
\end{align*}
where 

\begin{equation*}
d_\ell:=\min\{t\geq 0: \forall f\in M_{\frac{1}{2}(\ell-1)+\frac{c}{2}}\cap \mathbb{Z}_{(\ell)}\llbracket q\rrbracket,f\mid D_c(\ell)\mid X_c(\ell)^t\in S^{\text{odd}}(\ell,1)\} 
\end{equation*}  
and
\begin{equation*}
d_\ell':=\text{min}\{t\geq 0: \forall f\in M_{\frac{1}{2}(\ell-1)+\frac{c}{2}}\cap q^{\left\lceil \frac{c(\ell^2-1)}{24\ell^2}\right\rceil} \mathbb{Z}_{(\ell)}\llbracket q\rrbracket,f\mid Y_c(\ell)^t\in S^{\text{even}}(\ell,1)\} 
\end{equation*}

\end{remark}

\medskip

\noindent{\bf Data Availibility:} This manuscript has no associated data.

\bibliography{sn-bibliography}
\end{document}